\documentclass[12pt,oneside,english]{amsart}

\usepackage[T1]{fontenc}
\usepackage[utf8]{inputenc}
\usepackage[a4paper]{geometry}
\usepackage{babel}
\usepackage{mathtools}
\usepackage{amstext}
\usepackage{amsthm}
\usepackage{amssymb}
\usepackage[all]{xy}
\usepackage[unicode=true,
 bookmarks=true,bookmarksnumbered=false,bookmarksopen=false,
 breaklinks=false,pdfborder={0 0 1},backref=false,colorlinks=false]
 {hyperref}
\hypersetup{pdftitle={Roe functors preserve homotopies}}

\makeatletter
\numberwithin{equation}{section}
\numberwithin{figure}{section}
\theoremstyle{plain}
\newtheorem{thm}{\protect\theoremname}[section]
\theoremstyle{definition}
\newtheorem{defn}[thm]{\protect\definitionname}
\theoremstyle{remark}
\newtheorem{rem}[thm]{\protect\remarkname}
\theoremstyle{plain}
\newtheorem{lem}[thm]{\protect\lemmaname}

\usepackage{graphicx}

\usepackage[right]{lineno}
\usepackage[shortcuts]{extdash}
\usepackage{enumitem}

\usepackage[x11names]{xcolor}
\definecolor{cite_color}{rgb}{0.00, 0.55, 0.00}
\definecolor{link_color}{rgb}{0.55, 0.0, 0.0}

\hypersetup{
	pdfborder=0 0 0.1
}

\makeatother

\providecommand{\definitionname}{Definition}
\providecommand{\lemmaname}{Lemma}
\providecommand{\remarkname}{Remark}
\providecommand{\theoremname}{Theorem}

\newcommand{\norm}[1]{\left\Vert #1\right\Vert }
\newcommand{\ev}{\operatorname{ev}}
\newcommand{\mat}{\operatorname*{mat}}
\newcommand{\dist}{\operatorname{dist}}
\newcommand{\prop}{\operatorname{prop}}
\newcommand{\id}{\operatorname{id}}
\newcommand{\Id}{\operatorname{Id}}
\newcommand{\itId}{\mathit{Id}}
\newcommand{\Rot}{\operatorname{Rot}}
\newcommand{\close}{\sim_{cl}}
\newcommand{\crsh}{\simeq_{crs}}
\newcommand{\EFC}{\mathtt{EFC}}
\newcommand{\GEFC}{\mathtt{GEFC}}
\newcommand{\hGEFC}{\mathtt{hGEFC}}
\newcommand{\DMBG}{\mathtt{DMBG}}
\newcommand{\cDMBG}{\mathtt{cDMBG}}
\newcommand{\Cstar}{\mathtt{C^{*}}}

\author[G.S.~Makeev]{Georgii S.~Makeev}
\email{makeev.gs@gmail.com}

\date{}

\begin{document}

\title{Roe functors preserve homotopies}

\begin{abstract}
We show that coarse maps between countable metric spaces of bounded
geometry induce natural transformations of sufficiently good endofunctors
of $C^{*}$-algebras and prove that this correspondence is invariant
with respect to coarse homotopies.
\end{abstract}

\maketitle

\section{Introduction}

Many constructions involving homotopies of asymptotic homomorphisms
in the $E$-theory of Connes and Higson~\cite{connes-higson1990,GHT}
can be formulated at the level of endofunctors of $C^{*}$-algebras.
The category $\hGEFC$ introduced in~\cite{makeev_unsuspended} with
objects sufficiently good endofunctors of $C^{*}$-algebras happens
to be an appropriate category to describe this behavior.

Objects of the category $\hGEFC$ give generalized notions of homotopy
of $*$-homomorphisms. These homotopy relations can be used for constructing
an unsuspended picture of $E$-theory~\cite{makeev_unsuspended}.
They also give an $E$-theoretical analog of the extension groups
of $C^{*}$-algebras and can be used for computing $K$-homology.

In this work we study the relationship between discrete metric spaces
of bounded geometry with bounded coarse structure and Roe algebra
functors. We show that coarse maps induce morphisms in $\hGEFC$,
and this correspondence is invariant with respect to coarse homotopies.

\section{Roe algebras\label{sec:Definitions}}
\begin{defn}
A discrete metric space $X$ has \textit{bounded geometry}~\cite{roe_coarse_lectures}
if for every $R>0$ all the $R$-balls have uniformly bounded cardinalities,
i.e.
\[
\sup_{x\in X}\left|B_{R}(x)\right|<\infty.
\]
\end{defn}

We shall consider formal $X\textrm{-by-}X$-matrices with entries
in a $C^{*}$-algebra, and write $\mat_{x,y}b_{x,y}$ for the matrix
with elements $\{b_{x,y}\}_{x,y\in X}$. For such formal matrices
we can define the notion of \emph{propagation}:
\[
\prop\left(\mat_{x,y}b_{x,y}\right)\coloneqq\sup\left\{ \dist(x,y)\:\Big\vert\:x,y\in X,\:b_{x,y}\neq0\right\} .
\]

\begin{rem}
\label{rem:propmult}Note that propagation satisfies the following
property: 
\[
\prop(m_{1}m_{2})\leq\prop(m_{1})\prop(m_{2}).
\]
\end{rem}

Let $\mathcal{L}_{B}(l_{2}(X)\otimes B)$ be the $C^{*}$-algebra
of adjointable operators~\cite{lance1995} on the Hilbert $B$-module
$l_{2}(X)\otimes B$, and let 
\[
L(B,X)\coloneqq\left\{ m=\mat_{x,y}b_{x,y}\:\Big\vert\;b_{x,y}\in B,\;\prop(m)<\infty,\;\sup_{x,y\in X}\norm{b_{x,y}}<\infty\right\} .
\]
The two following lemmas are quite simple and can be proved by arguments
similar to those given in~\cite{makeev_2019}.
\begin{lem}
\label{lem:norm-ineq}Let $X$ be a discrete metric space of bounded
geometry, and let $m=\mat_{x,y\in X}b_{x,y}\in L(B,X)$. Then $m\in\mathcal{L}_{B}(l_{2}(X)\otimes B)$
with respect to the standard basis of $l_{2}(X)\otimes B$. Moreover,
there is $N>0$, which depends only on $X$ and $\prop(m)$, such
that the following inequality holds:
\[
\norm m\leq N\sup_{x,y\in X}\norm{b_{x,y}}.
\]
\end{lem}

\begin{defn}
\label{def:uRoe}We denote the norm closure of $L(B,X)$ in $\mathcal{L}_{B}(l_{2}(X)\otimes B)$
by $\mathfrak{M}_{X}^{u}B$, and call it the \textit{uniform Roe algebra}
of $X$ with coefficients in~$B$.
\end{defn}

\begin{lem}
$\mathfrak{M}_{X}^{u}$ is an endofunctor of the category of $C^{*}$-algebras
and $*$-homomorphisms.
\end{lem}

\begin{defn}
Let $\mathbb{K}$ be the functor of tensoring with the $C^{*}$-algebra
of compact operators in a separable Hilbert space. We call the composition
$\mathfrak{M}_{X}^{u}\mathbb{K}$ the \textit{Roe algebra functor
of X,} and denote it by~$\mathfrak{M}_{X}$.
\end{defn}

\section{Homotopies of natural transformations}

Denote by $\Cstar$ the category of $C^{*}$-algebras and $*$-homomorphisms,
and by $\EFC$ the category of endofunctors of $\Cstar$ and natural
transformations between them.

Let $\alpha\colon F\Rightarrow G$ and $\beta\colon G\Rightarrow H$
be two natural transformations. We denote their vertical composition
by $\beta\circ\alpha\colon F\Rightarrow H$. The horizontal composition
of $\alpha\colon F_{1}\Rightarrow F_{2}$ and $\beta\colon G_{1}\Rightarrow G_{2}$
will be denoted by $\alpha\beta\colon F_{1}G_{1}\Rightarrow F_{2}G_{2}$.

We write $\alpha B\colon FB\to GB$ for the component of $\alpha\colon F\Rightarrow G$
at object $B$. The symbol $\itId$ will denote the identity functor.
The expressions $c\in\in C$ and $f\in C$ will indicate respectively
that $c$ is an object and $f$ is an arrow of the category $C$.
\begin{defn}
\label{def:goodFinctor}We call $G\in\in\EFC$ a \textit{good endofunctor}
if the following properties hold:
\begin{enumerate}
\item \label{enu:h1}$G$ preserves epimorphisms;
\item \label{enu:h2}Let $\varphi_{1},\varphi_{2}\colon A\to B$ be two
$*$-homomorphisms, with $\varphi_{1}$ being an epimorphism; and
let 
\[
\xymatrix{B_{1}\oplus_{B}B_{2}\ar[d]_{p_{1}}\ar[r]^{p_{2}} & B_{2}\ar@{->>}[d]^{\varphi_{1}}\\
B_{1}\ar[r]_{\varphi_{2}} & B
}
\]
be the pullback diagram in $\Cstar$, where $p_{1}$, $p_{2}$ are
projections onto the corresponding components. Then the $*$-homomorphism
\[
G\left(B_{1}\underset{B}{\oplus}B_{2}\right)\to GB_{1}\underset{GB}{\oplus}GB_{2}
\]
induced by $Gp_{1}$ and $Gp_{2}$ is an isomorphism;
\item \label{enu:h3}There is a natural transformation $IG\Rightarrow GI$
making the following diagrams commute in~$\EFC$:
\[
\xymatrix{IG\ar@2[rr]\ar@2[rd]_{\ev_{j}G} &  & GI\ar@2[dl]^{G\ev_{j}}\\
 & G, &  &  & j=0,1,
}
\]
where $I$ is the functor of tensoring with $C[0,1]$, and where $\ev_{t}\colon I\Rightarrow\Id$
is the evaluation at~$t\in[0,1]$.
\end{enumerate}
\end{defn}

It is not difficult to prove that $I$, $\mathbb{K}$ and $\mathfrak{M}_{X}^{u}$
defined above are good endofunctors. Another important example which
we shall not, however, need further, but which provided motivation
for Definition~\ref{def:goodFinctor}, is the asymptotic algebra
functor $\mathfrak{A}$ from~\cite[Definition 1.1]{GHT}.

Good endofunctors and natural transformations form a category, which
we denote by $\GEFC$. It is easily verified that the composition
of two good endofunctors is again a good endofunctor. Hence, similarly
to $\EFC$, we can regard $\GEFC$ as a $2$-category~\cite{mclane-categories}.
Now we are ready to introduce homotopies of natural transformations
exposed in~\cite{makeev_unsuspended}.
\begin{defn}
We call natural transformations $\gamma_{0},\gamma_{1}\in\GEFC(F,G)$
\textit{homotopic} (written $\gamma_{0}\simeq\gamma_{1}$) if there
is a natural transformation $\gamma\colon F\Rightarrow GI$ such that
the following diagrams commute:
\[
\xymatrix{F\ar@2[dr]_{\gamma_{j}}\ar@2[r]^{\gamma} & GI\ar@2[d]^{G\ev_{j}}\\
 & G, & j=0,1.
}
\]
\end{defn}

One can check~\cite{makeev_unsuspended} that homotopy is an equivalence
relation, and that the vertical and horizontal compositions respect
homotopy. Hence, good endofunctors and homotopy classes of natural
transformations between them form a category, which we denote by $\hGEFC$. 
\begin{rem}
As the horizontal composition in $\hGEFC$ is well defined, we can
also regard $\hGEFC$ as a $2$-category.
\end{rem}

Consider the two natural transformations: a corner embedding $\iota\colon\Id\Rightarrow\mathbb{K}$
and a stability isomorphism $\theta\colon\mathbb{K}^{2}\Rightarrow\mathbb{K}$.
The proof of the following lemma is a standard argument, which can
be found for instance in~\cite{Jensen-Thomsen}.
\begin{lem}
\label{lem:corner}The diagram
\[
\xymatrix{\mathbb{K}\ar@{=}[dr]\ar@2[r]^{\iota\mathbb{K}} & \mathbb{K}^{2}\ar@2[d]^{\theta}\\
 & \mathbb{K}
}
\]
commutes up to homotopy.
\end{lem}

\section{Coarse discrete metric spaces}

In this section we remind some standard notions from coarse geometry~\cite{roe_coarse_lectures}
and introduce coarse homotopies, following~\cite{mitchener2020coarse}.

Let $f\colon X\to Y$ be a map between two metric spaces.

\begin{defn}
$f$ is \textit{proper} if $f^{-1}(B)$ is bounded for all bounded
subsets $B\subset Y$.
\end{defn}

\begin{defn}
\label{def:bornologous}$f$ is \textit{bornologous} if for all $N>0$
there is $M>0$ such that if $x,y\in X$ and $\dist(x,y)<N$, then
$\dist(f(x),f(y))<M$.
\end{defn}

\begin{defn}
$f$ is \textit{coarse} if it is proper and bornologous.
\end{defn}

\begin{defn}
Two maps $f,g\colon X\to Y$ are \textit{close} ($f\close g$) if
\[
\sup_{x\in X}\dist(f(x),g(x))<\infty.
\]
\end{defn}

\begin{defn}
A family of maps $\left\{ f_{n}\colon X\to Y\right\} _{n=1}^{\infty}$
is\textit{ equibornologous} if for all $N>0$ there is $M>0$ such
that if $\dist(x,y)<N$, then $\sup_{n}\dist(f_{n}(x),f_{n}(y))<M$.
\end{defn}

\begin{defn}
We denote by $\DMBG$ the category which has as objects countable
discrete metric spaces of bounded geometry and as morphisms coarse
maps.
\end{defn}

\begin{rem}
\label{rem:cdmbg}Denote by $\cDMBG$ the subcategory of $\DMBG$,
obtained by regarding two close maps as the same morphism. Actually,
$\cDMBG$ is just a full subcategory of the category of coarse spaces
spanned by countable discrete metric spaces of bounded geometry with
bounded coarse structure~\cite{roe_coarse_lectures}.
\end{rem}

\begin{defn}
Let $X\in\in\DMBG$, and let $p\colon X\to\mathbb{N}$ be a coarse
map. Define the $p$\textit{-cylinder}
\[
I_{p}X=\left\{ (x,n)\in X\times\mathbb{N}:n\leq p(x)+1\right\} 
\]
with the metrics given by the formula
\[
\dist((x,n),(y,m))=\dist(x,y)+|n-m|.
\]
We also introduce the two inclusions
\begin{align*}
i_{0} & \colon X\to I_{p}X\colon x\mapsto(x,1),\\
i_{1} & \colon X\to I_{p}X\colon x\mapsto(p(x)+1).
\end{align*}
Clearly, $I_{p}X\in\in\DMBG$ and $i_{0},i_{1}\in\DMBG$.\\
\end{defn}

\begin{defn}
Let $f_{0},f_{1}\in\DMBG(X,Y)$. We call them \textit{coarsely homotopic}
(written $f_{0}\crsh f_{1}$) if there is a $p$-cylinder $I_{p}X$
and a coarse map $H\colon I_{p}X\to Y$ making the diagrams
\[
\xymatrix{X\ar[d]_{i_{k}}\ar[rrd]^{f_{k}}\\
I_{p}X\ar[rr]^{H} &  & Y, & \qquad k=0,1
}
\]
commute. The map $H$ is called a \textit{coarse homotopy}.
\end{defn}

\section{Natural transformation induced by coarse maps}

We now move on to study Roe algebra functors as objects of $\hGEFC$.

For a countable metric space $X$ we shall sometimes write $\mathbb{K}_{X}$
instead of $\mathbb{K}$ to emphasize that we use $l_{2}(X)$ as the
Hilbert space in the definition of~$\mathbb{K}$. By $\theta_{X}\colon\mathbb{K}_{X}\mathbb{K}\Rightarrow\mathbb{K}$
we denote a stability isomorphism. We also specify the ``corner''
embedding $\iota_{x_{0}}\colon\Id\Rightarrow\mathbb{K}_{X}$ by the
formula $b\mapsto b\otimes\epsilon_{x_{0},x_{0}}$, where $x_{0}\in X$
is an arbitrary point and $\epsilon_{x_{0},x_{0}}$ is the corresponding
matrix unit.
\begin{defn}
Let $f\in\DMBG(X,Y)$. Consider the natural transformation
\[
f_{*}\colon\mathfrak{M}_{X}^{u}\Rightarrow\mathfrak{M}_{Y}^{u}\mathbb{K}_{X},
\]
defined at object $B$ by the formula
\[
f_{*}B\colon\mat_{x_{1},x_{2}}b_{x_{1},x_{2}}\mapsto\mat_{y_{1},y_{2}}m_{y_{1},y_{2}},
\]
where
\[
m_{y_{1},y_{2}}=\mat_{x_{1},x_{2}\in X}\left(\begin{cases}
b_{x_{1},x_{2}}, & \textrm{if }f(x_{1})=y_{1}\textrm{ and }f(x_{2})=y_{2};\\
0, & \textrm{otherwise.}
\end{cases}\right)
\]
Note that $m_{y_{1},y_{2}}\in\mathbb{K}_{X}B$, for it has only finitely
many nonzero elements due to properness of~$f$.
\end{defn}

\begin{defn}
We denote by $\mathfrak{M}_{f}$ the natural transformation given
by
\[
\mathfrak{M}_{f}\colon\mathfrak{M}_{X}=\mathfrak{M}_{X}^{u}\mathbb{K}\xRightarrow{f_{*}\mathbb{K}}\mathfrak{M}_{Y}^{u}\mathbb{K}_{X}\mathbb{K}\xRightarrow{\mathfrak{M}_{Y}^{u}\theta_{X}}\mathfrak{M}_{Y}^{u}\mathbb{K}=\mathfrak{M}_{Y}.
\]
The homotopy class of $\mathfrak{M}_{f}$ does not depend on the choice
of $\theta_{X}$, because any two such isomorphisms are homotopic.
\end{defn}

\begin{rem}
\label{rem:Mf-prop-incr}Let $m\in\mathfrak{M}_{X}^{u}B$ be a matrix
of finite propagation. Note that if $\prop(m)<N$ then $\prop(\mathfrak{M}_{f}B(m))<M$,
where $N$ and $M$ are the numbers from Definition~\ref{def:bornologous}.
\end{rem}

\begin{defn}
Let $X\in\in\DMBG$ and let $\sigma\colon X\to X$ be a bijection
such that $\sigma\circ\sigma=\id$. Define $\Rot_{\sigma}(t)$ to
be a standard family of (unitary) rotation matrices swapping $x$-th
and $\sigma(x)$-th basis elements in $l_{2}(X)$. In order to explicitly
define it, endow $X$ with an order by enumerating its elements, and
take
\[
\Rot_{\sigma}\colon[0,1]\to\mathbb{B}(l_{2}(X))\colon t\mapsto\mat_{x_{1},x_{2}\in X}\left(\begin{cases}
1, & x_{1}=x_{2}\textrm{ and }\sigma(x_{1})=x_{1};\\
\cos(\pi t/2), & x_{1}=x_{2}\textrm{ and }\sigma(x_{1})\neq x_{1};\\
\sin(\pi t/2), & x_{2}>x_{1}\textrm{ and }\sigma(x_{1})=x_{2};\\
-\sin(\pi t/2), & x_{2}<x_{1}\textrm{ and }\sigma(x_{1})=x_{2};\\
0, & \textrm{otherwise.}
\end{cases}\right).
\]
\end{defn}

\begin{lem}
\label{lem:rot}Let $X$ and $\sigma$ be as in the previous definition.
Then the map $t\mapsto\Rot_{\sigma}(t)$ is strictly continuous. If
in addition $\sigma\close\id$, then $\Rot_{\sigma}$ is norm-continuous
and 
\[
\prop(\Rot_{\sigma})=\sup_{x\in X}\dist(x,\sigma(x)).
\]
\end{lem}

\begin{proof}
The first part of the lemma is obvious. The second part follows from
Lemma~\ref{lem:norm-ineq}.
\end{proof}
\begin{lem}
\label{lem:dmbg-functoriality}The correspondence $f\mapsto\mathfrak{M}_{f}$
is a functor from $\DMBG$ to $\hGEFC$.
\end{lem}

\begin{proof}
Let $f\in\DMBG(X,Y)$, $g\in\DMBG(Y,Z)$, and let $y_{0}$ be an arbitrary
point in $Y$. To prove that $\mathfrak{M}_{g\circ f}\simeq\mathfrak{M}_{g}\circ\mathfrak{M}_{f}$
it suffices to check that the following diagram commutes in $\hGEFC$:
\[
\xymatrix{ &  & \mathfrak{M}_{X}^{u}\mathbb{K}\ar@2[ddll]_{f_{*}\mathbb{K}}\ar@2[ddrr]^{(g\circ f)_{*}\mathbb{K}}\\
\\
\mathfrak{M}_{Y}^{u}\mathbb{K}_{X}\mathbb{K}\ar@2[rr]^{g_{*}\mathbb{K}_{X}\mathbb{K}}\ar@2[dd]^{\mathfrak{M}_{Y}^{u}\theta_{X}} &  & \mathfrak{M}_{Z}^{u}\mathbb{K}_{Y}\mathbb{K}_{X}\mathbb{K}\ar@2[dd]^{\mathfrak{M}_{Z}^{u}\mathbb{K}_{Y}\theta_{X}} &  & \mathfrak{M}_{Z}^{u}\mathbb{K}_{X}\mathbb{K}\ar@2[ll]_{\mathfrak{M}_{Z}^{u}\iota_{y_{0}}\mathbb{K}_{X}\mathbb{K}}\ar@2[dd]^{\mathfrak{M}_{Z}^{u}\theta_{X}}\\
\\
\mathfrak{M}_{Y}^{u}\mathbb{K}\ar@2@2[rr]^{g_{*}\mathbb{K}} &  & \mathfrak{M}_{Z}^{u}\mathbb{K}_{Y}\mathbb{K}\ar@2[rrdd]_{\mathfrak{M}_{Z}^{u}\theta_{Y}} &  & \mathfrak{M}_{Z}^{u}\mathbb{K}\ar@2[ll]_{\mathfrak{M}_{Z}^{u}\iota_{y_{0}}\mathbb{K}}\ar@{=}[dd]\\
\\
 &  &  &  & \mathfrak{M}_{Z}^{u}\mathbb{K}.
}
\]
Commutativity of the two rectangles in the middle is obvious. The
lower triangle commutes by Lemma~\ref{lem:corner}. It remains only
to prove that the upper part of the diagram commutes.

To this end let us consider the bijection $\sigma\colon Y\times X\to Y\times X$
swapping the points $(f(x),x)$ and $(y_{0},x)$ for all $x\in X$.
By the first part of Lemma~\ref{lem:rot} we have a strictly continuous
path $\Rot_{\sigma}$ in the multipliers of $\mathbb{K}_{Y}\mathbb{K}_{X}\mathbb{K}$,
which yields
\[
\mathfrak{M}_{Z}^{u}\mathbb{K}_{Y}\mathbb{K}_{X}\mathbb{K}\Rightarrow\mathfrak{M}_{Z}^{u}I\mathbb{K}_{Y}\mathbb{K}_{X}\mathbb{K}\colon\mat_{z_{1},z_{2}}k_{z_{1},z_{2}}\mapsto\mat_{z_{1},z_{2}}[t\mapsto\Rot_{\sigma}(t)k_{z_{1},z_{2}}\Rot_{\sigma}(t)^{*}].
\]
The homotopy $\mathfrak{M}_{\id}\simeq\id$ is obtained by an analogous
argument.
\end{proof}
\begin{lem}
\label{lem:close-to-homotopy}Let $f,g\in\DMBG(X,Y)$ such that $f\close g$.
Then there is a homotopy $\eta\colon\mathfrak{M}_{X}\Rightarrow\mathfrak{M}_{Y}I$
connecting $\mathfrak{M}_{f}$ and $\mathfrak{M}_{g}$ such that
\begin{equation}
\prop(\eta B(m))\leq\prop(\mathfrak{M}_{f}B(m))+2\sup_{x\in X}\dist(f(x),g(x)).\label{eq:homotopy-prop}
\end{equation}
\end{lem}

\begin{proof}
Let $\sigma\colon Y\times X\to Y\times X$ be a bijection swapping
the points $(f(x),x)$ and $(g(x),x)$ for all $x\in X$. We apply
the second part of Lemma \ref{lem:rot} to obtain a norm-continuous
path of unitaries $\Rot_{\sigma}$ such that
\begin{equation}
\prop(\Rot_{\sigma})\leq\sup_{x\in X}\dist(f(x),g(x)).\label{eq:rotprop}
\end{equation}
We use Axiom~\ref{enu:h3} to derive the required homotopy from the
natural transformation
\[
\mathfrak{M}_{X}\Rightarrow I\mathfrak{M}_{Y}\colon m\mapsto[t\mapsto\Rot_{\sigma}(t)m\Rot_{\sigma}(t)^{*}].
\]
The inequality~(\ref{eq:homotopy-prop}) holds due to Remark~\ref{rem:propmult}
and inequality~(\ref{eq:rotprop}).
\end{proof}
\begin{rem}
\label{rem:constant}Let $m\in\mathfrak{M}_{X}B$, $y_{1},y_{2}\in Y$
and let $\eta\colon\mathfrak{M}_{X}\Rightarrow\mathfrak{M}_{Y}I$
be the homotopy from the previous lemma. It is easy to see that if
$f(x)=g(x)$ for all 
\[
x\in f^{-1}\{y_{1},y_{2}\}\cup g^{-1}\{y_{1},y_{2}\},
\]
then the matrix element $(\eta B(m))_{y_{1},y_{2}}$ is a constant
function.
\end{rem}

Let $\cDMBG$ be the category from Remark~\ref{rem:cdmbg}. We have
just proved the following theorem.
\begin{thm}
The correspondence $f\mapsto\mathfrak{M}_{f}$ is a functor from $\cDMBG$
to $\hGEFC$.
\end{thm}

The proof of homotopy invariance of this functor relies on the following
technical lemma.
\begin{lem}
\label{lem:family}Let $H\colon I_{p}X\to Y$ be a coarse homotopy
between $f$ and $g$. Then there is a family of coarse maps $\left\{ f_{n}\colon X\to Y\right\} _{n\in\mathbb{N}}$
such that the following properties hold:
\begin{enumerate}
\item \label{enu:unibor}the family $\left\{ f_{n}\right\} $ is equibornologous;
\item \label{enu:close}$\sup_{n\in\mathbb{N}}\sup_{x\in X}\dist(f_{n}(x),f_{n+1}(x))<\infty$;
\item \label{enu:stationary}for all $x\in X$ there is $N\in\mathbb{N}$
such that $f_{n}(x)=f_{N}(x)$ for all $n\geq N$;
\item \label{enu:eventually}for every $y\in Y$ there are only finitely
many $x\in X$ such that $\{f_{n}(x)\}_{n\in\mathbb{N}}\cap H^{-1}\{y\}\neq\text{Ø}$;
\item \label{enu:coarselim}$f_{1}=f$ and $\lim_{n\to\infty}f_{n}(x)=g$.
\end{enumerate}
\end{lem}

\begin{proof}
We define $f_{n}$ by the formula
\[
f_{n}\colon X\to Y\colon x\mapsto H(x,\min(n,p(x)+1)).
\]
Bearing in mind that $p$ is bornologous, we use the inequality
\begin{multline*}
\dist\left((x,\min(n,p(x)+1)),(y,\min(n,p(y)+1))\right)=\\
\dist(x,y)+\dist\left(\min(n,p(x)+1),\min(n,p(y)+1)\right)\leq\dist(x,y)+|p(x)-p(y)|
\end{multline*}
to show that~\ref{enu:unibor} holds. Properties~\ref{enu:close}
and~\ref{enu:eventually} follow from the fact that $H$ is a coarse
map. Property~\ref{enu:stationary} is obvious. Note that $f_{1}$
and $\lim_{n\to\infty}f_{n}(x)$ are the restrictions of $H$ to the
upper and lower faces of $I_{p}$, respectively, so~\ref{enu:coarselim}
also holds.
\end{proof}
\begin{thm}
\label{thm:homotopic2homotopic}If $f\crsh g\colon X\to Y$, then
$\mathfrak{M}_{f}\simeq\mathfrak{M}_{g}$. 
\end{thm}

\begin{proof}
Let $\{f_{n}\}$ be a sequence of maps given by Lemma~\ref{lem:family}.
From~\ref{enu:close} we deduce that $f_{n}\close f_{n+1}$ for all
$n$. Using this we get a family of natural transformations $\{\eta_{n}\colon\mathfrak{M}_{X}\Rightarrow I\mathfrak{M}_{Y}\}_{n\in\mathbb{N}}$
as in the proof of Lemma~\ref{lem:close-to-homotopy} such that $\ev_{0}\circ\eta_{n}=\mathfrak{M}_{f_{n}}$
and $\ev_{1}\circ\eta_{n}=\mathfrak{M}_{f_{n+1}}$. Putting them together
in an obvious way, we obtain the natural transformation
\[
\eta\colon\mathfrak{M}_{X}\Rightarrow C_{b}[1,+\infty)\otimes\mathfrak{M}_{Y},
\]
where $C_{b}[1,+\infty)$ is the $C^{*}$-algebra of bounded functions
on $[1,+\infty)$. By Properties~\ref{enu:unibor}, \ref{enu:close},
Remark~\ref{rem:Mf-prop-incr}, and Lemma~\ref{lem:close-to-homotopy}
we may deduce that actually
\[
\eta\colon\mathfrak{M}_{X}\Rightarrow\mathfrak{M}_{Y}C_{b}[1,+\infty).
\]
To see that $\eta$ is the required homotopy it suffices to check
that for all $B\in\in\Cstar$ and $m\in\mathfrak{M}_{X}B$, all the
elements of the matrix $\eta B(m)$ are functions constant at infinity.

Let $y_{1},y_{2}\in Y$. We use Properties~\ref{enu:eventually},~\ref{enu:stationary}
to find a finite subset $X'\subset X$ and $N\in\mathbb{N}$ such
that for all $n\geq N$
\begin{enumerate}
\item $f_{n}(x)=f_{N}(x)$ for all $x\in X'$;
\item $f_{n}(x)\in\{y_{1},y_{2}\}$ iff $x\in X'$.
\end{enumerate}
From Remark~\ref{rem:constant} it follows that the function $t\mapsto(\eta B(m))_{y_{1},y_{2}}(t)$
is constant for $t\geq N$.
\end{proof}


\begin{thebibliography}{1}

\bibitem{connes-higson1990}
A.~Connes and N.~Higson.
\newblock D{\'e}formations, morphismes asymptotiques et ${K}$-th{\'e}orie bivariante.
\newblock {\em CR Acad. Sci. Paris S{\'e}r. I Math}, 311(2):101--106, 1990.

\bibitem{GHT}
E.~Guentner, N.~Higson, and J.~Trout.
\newblock Equivariant ${E}$-theory for ${C^*}$-algebras.
\newblock {\em Mem. Amer. Math. Soc.}, 148(703), 11 2000.

\bibitem{Jensen-Thomsen}
K.~K. Jensen and K.~Thomsen.
\newblock {\em Elements of ${KK}$-theory}.
\newblock Birkh{\"a}user, Boston, 1991.

\bibitem{lance1995}
E.~C. Lance.
\newblock {\em Hilbert ${C^*}$-modules: a toolkit for operator algebraists}.
\newblock Cambridge University Press, 1995.

\bibitem{makeev_2019}
G.~S. Makeev.
\newblock Yet another descriptione description of the {C}onnes-{H}igson functor.
\newblock {\em Math. Notes}, 107(2):97--108, 2020.

\bibitem{makeev_unsuspended}
G.~S. Makeev.
\newblock An unsuspended description of the ${E}$-theory category.
\newblock {\em Moscow Univ. Math. Bull.}, 78(1):1--14, 2023.

\bibitem{mclane-categories}
S.~McLane.
\newblock {\em Categories for the working mathematician}.
\newblock Springer, New York, 1998.

\bibitem{mitchener2020coarse}
P.~D. Mitchener, B.~Norouzizadeh, and T.~Schick.
\newblock Coarse homotopy groups.
\newblock {\em Math. Nachr.}, 293(8):1515--1533, 2020.

\bibitem{roe_coarse_lectures}
J.~Roe.
\newblock {\em Lectures on coarse geometry}.
\newblock Number~31. Amer. Math. Soc., 2003.

\end{thebibliography}
\end{document}